\documentclass[a4paper]{article}

\usepackage[mathscr]{eucal}
\usepackage{amssymb}
\usepackage{latexsym}
\usepackage{amsthm}
\usepackage{amsmath}
\usepackage[dvips]{graphicx}
\usepackage{psfrag}

\DeclareMathOperator{\Gal}{Gal}
\newtheorem{theorem}{Theorem}[section]
\newtheorem{proposition}[theorem]{Proposition}
\newtheorem{lemma}[theorem]{Lemma}

\theoremstyle{definition}

\newtheorem{remark}[theorem]{Remark}

\makeatletter
\@addtoreset{equation}{section}

\makeatother

\title{A characterization of $Q$-polynomial association schemes} 
\author{Hirotake Kurihara, Hiroshi Nozaki}

\begin{document}
\maketitle

\renewcommand{\thefootnote}{\fnsymbol{footnote}}
\footnote[0]{2010 Mathematics Subject Classification: 05E30 (51D20).}
\begin{abstract}
We prove a necessary and sufficient condition for a symmetric association
scheme to be a $Q$-polynomial scheme.
\end{abstract}

\textbf{Key words}: $Q$-polynomial association scheme,
$s$-distance set.

\section{Introduction} 
	A {\it symmetric association scheme} of class $d$ is a pair $\mathfrak{X} = (X, \{R_i \}_{i=0}^d)$, 
	where $X$ is a finite set and each $R_i$ is a nonempty subset of $X \times X$ satisfying the following: 
	\begin{enumerate}
	\item $R_0 = \{(x, x) \mid x \in X \}$,
	\item $X \times X = \bigcup_{i=0}^d R_i$ and $R_i \cap R_j$ is empty if $i\ne j$, 
	\item $^tR_i = R_i$ for any $i \in\{ 0, 1,\ldots , d\}$, where $^t R_i = \{(y, x) \mid (x, y) \in R_i\}$, 
	\item for all $i,j,k\in \{0,1,\ldots ,d\}$, there exist integers $p_{i j}^k$ such that for all $x,y \in X$ with $(x,y) \in R_k$, 
	\[ p_{i j}^k= |\{ z \in X \mid (x,z) \in R_i, (z,y) \in R_j \}|.  \]
	\end{enumerate}
	The integers $p_{i j}^k$ are called the {\it intersection numbers}. 
	
	Let $\mathfrak{X}$ be a symmetric association scheme.
	The $i$-th {\it adjacency matrix} $A_i$ of $\mathfrak{X}$ is the matrix with rows and columns indexed by $X$ 
	such that the $(x,y)$-entry is $1$ if $(x,y)\in R_i$ or $0$ otherwise.  
	The {\it Bose--Mesner algebra} of $\mathfrak{X}$ is the algebra generated 
	by the adjacency matrices $\{A_i \}_{i=0}^d$ over the complex field $\mathbb{C}$. 
	Then $\{A_i\}_{i=0}^d$ is a natural basis of the Bose--Mesner algebra. 
	By \cite[page 59]{Bannai-Ito}, the Bose--Mesner algebra 
	has a second basis $\{E_i\}_{i=0}^d$ such that 
	\begin{enumerate}
	\item $E_0=|X|^{-1}J$, where $J$ is the all-ones matrix, 
	\item $I=\sum_{i=0}^d E_i$, where $I$ is the identity matrix, 
	\item $E_i E_j= \delta_{ij}E_i$, where $\delta_{ij}=1$ if $i =j$ and $\delta_{ij}=0$ if $i \ne j$. 
	\end{enumerate}
	The basis $\{E_i\}_{i=0}^d$ is called the {\it primitive idempotents} of $\mathfrak{X}$. 
	We have the following equations: 
	\begin{align}
	A_i&= \sum_{j=0}^d p_i(j)E_j, \label{eq:basic} \\
	E_i&=\frac{1}{|X|}\sum_{j=0}^d q_i(j) A_j, \\
	A_i A_j&= \sum_{k=0}^d p_{i j}^k A_k, \\ 
	E_i \circ E_j&= \frac{1}{|X|} \sum_{k=0}^d q_{i j}^k E_k,
	\end{align}
	where $\circ$ denotes the {\it Hadamard product}, that is, the entry-wise matrix product. 
	The matrices 
	$P=(p_j(i))^d_{i,j=0}$ and $Q=(q_j(i))^d_{i,j=0}$ are called the first and second {\it eigenmatrices}, respectively. 
	The numbers $q_{i j}^k$ are called the {\it Krein parameters}. 
	The Krein parameters are nonnegative real numbers (the Krein condition) 
	\cite{Scott} \cite[page 69]{Bannai-Ito}. 
	 
	A symmetric association scheme is called a {\it $P$-polynomial scheme} (or a {\it metric scheme}) 
	with respect to the ordering $\{A_i \}_{i=0}^d$ if for each $i \in \{0,1,\ldots, d \}$, 
	there exists a polynomial $v_i$ of degree $i$ such that $p_i(j)=v_i(p_1(j))$ for any $j \in \{0,1,\ldots, d \}$.
	We say a symmetric association scheme is 
a $P$-polynomial scheme with respect to $A_1$
if it
has the $P$-polynomial property
with respect to some ordering
$A_0,A_1, A_{i_2},A_{i_3},\ldots ,A_{i_d}$.
	Dually a symmetric association scheme is called a {\it $Q$-polynomial scheme} (or a {\it cometric scheme}) 
	with respect to the ordering $\{E_i \}_{i=0}^d$ if for each $i \in \{0,1,\ldots, d \}$, 
	there exists a polynomial $v_i^{\ast}$ of degree $i$ 
	such that $q_i(j)=v_i^{\ast}(q_1(j))$ for any $j \in \{0,1,\ldots, d \}$. 
	Moreover a symmetric association scheme is called
a $Q$-polynomial scheme with respect to $E_1$
if it
has the $Q$-polynomial property
with respect to some ordering
$E_0,E_1, E_{i_2},E_{i_3},\ldots ,E_{i_d}$.
Note that both $\{v_i\}^d_{i=0}$ and $\{v_i^{\ast}\}^d_{i=0}$ form systems of orthogonal polynomials.
	
	Throughout this paper, we use the notation $m_i=q_i(0)$ and $\theta_i^{\ast}=q_1(i)$ 
for $0 \leq i \leq d$. 
	 If an association scheme is $Q$-polynomial, then $\{ \theta_i^{\ast} \}_{i=0}^d$ are mutually distinct
because the second eigenmatrix $Q=(v_i^{\ast}(\theta_j^{\ast}))^d_{j,i=0}$ is non-singular.
For a univariate polynomial $f$ and a matrix $M$,
we denote by $f(M^\circ)$ the matrix obtained by substituting $M$
into $f$ with multiplication the Hadamard product.
	We introduce known equivalent conditions of the $Q$-polynomial property
of symmetric association schemes \cite[page 193]{Bannai-Ito}.
	The following are equivalent:
	\begin{enumerate}
	\item $\mathfrak{X}$ is a $Q$-polynomial scheme with respect to 
	the ordering $\{E_i \}_{i=0}^d$. 
	\item $(q_{1,i}^j)^d_{i,j=0}$ is an irreducible tridiagonal matrix.
	\item For each $i \in \{0,1,\ldots,d \}$, there exists a polynomial $f_i$ of degree $i$ 
	such that $E_i=f_i(E_1^{\circ})$. 
	\end{enumerate}
	
	In the present paper, we prove a new necessary and sufficient condition for a symmetric
association scheme to be $Q$-polynomial.
Since the $Q$-polynomial property of a symmetric association scheme of class $1$ is trivial, 
we assume that $d$ is greater than $1$.
	\begin{theorem} \label{main}
	 Let $\mathfrak{X}$ be a symmetric association scheme of class $d\ge 2$. 
Suppose that $\{\theta_j^{\ast}\}^d_{j=0}$ are mutually distinct. 
Then the following are equivalent: 
\begin{enumerate}
	\item $\mathfrak{X}$ is a $Q$-polynomial scheme with respect to $E_1$. 
	\item 
There exists $l \in \{2,3, \ldots ,d\}$ such that
for any $i \in \{1,2,\ldots, d\}$,
	\[
	 \prod_{\substack{j=1 \\ j \ne  i}}^d \frac{\theta_0^\ast-\theta_j^\ast}{\theta_i^\ast-\theta_j^\ast}= -p_i(l). 
	\]
\end{enumerate}
	Moreover if $(2)$ holds, then $l=i_d$.
	\end{theorem}
	
	\begin{remark}
	We call a finite set $X$ in $\mathbb{R}^m$ a {\it $d$-distance set} if the number of the Euclidean distances 
	between distinct two points in $X$ is equal to $d$. 
	Larman--Rogers--Seidel \cite{Larman-Rogers-Seidel} proved that if the size of a two-distance set with the distances $a,b$ ($a<b$) is greater
than $2m+3$, then there exists a positive integer $k$ such that $a^2/b^2=(k-1)/k$, 
{\it i.e.} $k=b^2/(b^2-a^2)$. 
Bannai--Bannai \cite{Bannai-Bannai} proved that the ratio $k$ of the spherical embedding of a primitive association scheme of class $2$ coincides with $-p_i(2)$. 
The research of the present paper is motivated by \cite{Bannai-Bannai}. 
For a symmetric association scheme satisfying that $\{\theta_j^{\ast}\}^d_{j=0}$ are mutually distinct,
the values $K_i:=\prod_{j=1, j \ne  i}^d (\theta_0^\ast-\theta_j^\ast)(\theta_i^\ast-\theta_j^\ast)^{-1}$ ($1\leq i \leq d$) are the 
generalized Larman--Rogers--Seidel's ratios \cite{Nozaki} of the spherical embedding of this association scheme with respect to $E_1$. 
Theorem~\ref{main} is an extension of Bannai--Bannai's result to $Q$-polynomial schemes of any class. 
Furthermore Theorem~\ref{main} is a new characterization of the $Q$-polynomial property on
 the spherical embedding of a symmetric association scheme.  
	\end{remark}
	At the end of this paper, we give some sufficient conditions for the integrality of $K_i$.

\section{Proof of Theorem \ref{main}}
	
	First we give several lemmas that will be needed to 
prove Theorem \ref{main}. 
	
	\begin{lemma} \label{sum_Ki_2}
	For any mutually distinct real numbers $\beta_1,\beta_2, \ldots, \beta_s$, 
the following identity holds.  
\[\sum_{i=1}^s \beta_i^j \prod_{\substack{k=1\\ k \ne  i}}^s \frac{x-\beta_k}{\beta_i-\beta_k}=x^j\] for any $j \in \{0,1, \ldots ,s-1\}$, 
where $x$ is a variable. 

	\end{lemma}
	
	\begin{proof}
	For each $j\in \{0,1,\ldots,s-1\}$, the polynomial 
\[L_j(x):=\sum^s_{i=1}\beta_i^j\prod_{\substack{k=1\\ k \ne  i}}^s\frac{x-\beta_k}{\beta_i-\beta_k}\] 
	of degree at most $s-1$ is known as the interpolation polynomial in the Lagrange form (see \cite{lag}). 
	Namely, the property $L_j(\beta_i)=\beta_i^j$ holds for any $i\in \{1,2,\ldots ,s\}$.  
	Therefore $L_j(x)=x^j$, and the lemma follows.
	\end{proof}
	We say $E_j$ is a {\it component} of an element $M$ of the Bose--Mesner algebra 
	if $E_jM\neq0$. 
	Let $N_h$ denote the set of indices $j$ such that $E_j$ is a component of $E_1^{\circ h}$ 
	but not of $E_1^{\circ l}$ $(0 \leq l \leq h-1)$.  Note that $N_0=\{0\}$ and $N_1=\{1\}$.  
	
	\begin{lemma} \label{key}
	Suppose $\mathfrak{X}$ is a symmetric association scheme of class $d \ge 2$. Then the following are equivalent. 
	\begin{enumerate}
	\item $\mathfrak{X}$ is a $Q$-polynomial scheme with respect to $E_1$. 
	\item The cardinality of $N_d$ is equal to $1$. 
	\item $N_d$ is nonempty. 
	\end{enumerate}
	\end{lemma}
	\begin{proof}
	$(2)\Rightarrow (3)$: Clear. \\
	$(1) \Rightarrow (2)$: 
	Without loss of generality, we assume that $\mathfrak{X}$ is a $Q$-polynomial scheme with respect to $\{E_i\}^d_{i=0}$.
	By noting that $\{ \theta_i^{\ast} \}_{i=0}^{d}$ are mutually distinct, 
$\{ E_1^{\circ i} \}_{i=0}^d$ are linearly independent, and a basis of the Bose--Mesner algebra. 
We have 
\[
E_i=f_i(E_1^{\circ})= \sum_{j=0}^i \alpha_{i,j} E_1^{\circ j}, 
\]
where $\alpha_{i,j}\in \mathbb{R}$ are the coefficients of a polynomial $f_i$ of degree $i$.
The upper triangular matrix $(\alpha_{i,j})^d_{i,j=0}$ is non-singular because $\alpha_{i,i} \ne 0$ for each $i$. 
Since the inverse matrix $(\alpha'_{i,j})^d_{i,j=0}$ of $(\alpha_{i,j})^d_{i,j=0}$ is also an upper triangular matrix with $\alpha'_{i,i} \ne 0$ for each $i$, 
we can express 
\[
E_1^{\circ i}= \sum_{j=0}^i \alpha'_{i,j} E_j. 
\] 
Therefore $(2)$ follows. \\
$(3) \Rightarrow (1)$:  
First we prove that if $N_{i}$ is empty for some $i\in\{1,2,\ldots,d-1\}$, then $N_{i+1}$ is also empty.  
Let $\mathcal{I}=\cup_{j=0}^{i-1} N_j$. 
We consider the expression  
$
\sum_{j =0}^{i-1} E_1^{\circ j} = \sum_{j \in \mathcal{I}} \beta_j E_j 
$.
Note that $\beta_j>0$ for any $j\in \mathcal{I}$ by the Krein condition. 
Then we have 
\[
E_1 \circ (\sum_{h =0}^{i-1} E_1^{\circ h}) = \sum_{j \in \mathcal{I}} \beta_j \sum_{k=0}^d q_{1,j}^k E_k
 = \sum_{k=0}^d \sum_{j \in \mathcal{I}} \beta_j q_{1,j}^k E_k.
\] 
If $N_{i}$ is empty, then 
\begin{equation} \label{q}
q_{1,j}^k=0 \text{ for any $j\in \mathcal{I}$ and any $k \not\in \mathcal{I}$}
\end{equation}
because $\beta_j>0$ holds for any $j\in \mathcal{I}$. 
We can express $E_1^{\circ i}= \sum_{j \in \mathcal{I}} \beta_j' E_j $,
where $\beta_j'$ are non-negative integers for any $j\in \mathcal{I}$.
By \eqref{q} and the equalities 
\[
E_1^{\circ (i+1)}=E_1 \circ E_1^{\circ i}= E_1 \circ \sum_{j \in \mathcal{I}} \beta_j' E_j=  \sum_{k=0}^d \sum_{j \in \mathcal{I}} \beta_j' q_{1,j}^k E_k, 
\]
we obtain $\sum_{j \in \mathcal{I}} \beta_j' q_{1,j}^k=0$ for $k \not\in \mathcal{I}$. 
Hence $N_{i+1}$ is also empty. 
	This means that if $N_d$ is not empty, then the cardinalities of $N_h$ is equal to $1$ for any $h \in \{0,1,\ldots,d \}$. 
	Put $N_h=\{i_h\}$ and order $E_0,E_1,E_{i_2},E_{i_3},\ldots, E_{i_d}$. 
	Then we can construct polynomials $f_h$ of degree $h$
such that $f_h(E_1^{\circ})=E_{i_h}$ for any $h \in \{0,1,\ldots, d\}$.
	Hence $(1)$ follows. 
	\end{proof}

Now we prove Theorem \ref{main}. 
\begin{proof}[Proof of Theorem \ref{main}]
(1) $\Rightarrow$ (2):
Without loss of generality, we assume that $\mathfrak{X}$ is a $Q$-polynomial scheme with respect to $\{E_i\}^d_{i=0}$.
For each $i \in \{1,2,\ldots ,d \}$, we define the polynomial   
	\[
	F_i(t):=\prod_{\substack{j=1\\ j \ne i}}^d \frac{|X|t- \theta^{\ast}_j }{\theta^{\ast}_i-\theta^{\ast}_j}
	\]
	of degree $d-1$. 
	Set $M_i=F_i(E_1^{\circ})$. Then $|X|E_1=\sum_{j=0}^d \theta_j^\ast A_j$ yields that the $(x,y)$-entries of $M_i$ are
	\[
	M_i(x,y)=\begin{cases}
	K_i \text{\qquad if $(x,y) \in R_0$},\\ 
	1 \text{\qquad if $(x,y) \in R_i$},\\
	0 \text{\qquad otherwise},
	\end{cases}
	\]
	where $K_i:=\prod_{j=1, j \ne  i}^d (\theta_0^\ast-\theta_j^\ast)(\theta_i^\ast-\theta_j^\ast)^{-1}$.
	Since $F_i$ is a polynomial of degree $d-1$, the matrix $M_i$ is a linear combination of $\{ E_i \}_{i=0}^{d-1}$. 
This means that $M_i E_d=0$. By \eqref{eq:basic}, 
\[
0=M_i E_d=(K_iI+A_i)E_d=(K_i+p_i(d))E_d  
\]
for any $i\in \{1,2,\ldots,d\}$. Therefore the desired result follows. 

(2) $\Rightarrow$ (1): 
From the equation $A_i= \sum_{j=0}^d p_i(j)E_j$ and our assumptions, we have 
\[A_iE_l= p_i(l)E_l=-K_iE_l.\] 
By Lemma \ref{sum_Ki_2}, 
	\[
	(|X|E_1)^{\circ j} E_l=\bigr( (\theta_0^\ast)^j I + \sum_{i=1}^d (\theta^{\ast}_i)^j A_i \bigr) E_l=
	\bigr( (\theta_0^\ast)^j  - \sum_{i=1}^d (\theta^{\ast}_i)^j K_i \bigr) E_l=0
	\]
	for any $j \leq d-1$. This means that $l$ is not an element of $N_j$ for any $j \leq d-1$. 
	Note that the following equality holds:  
	\begin{equation*} 
	\prod^{d}_{j=1}\frac{|X|E_1-\theta^\ast_j J}{\theta_0^\ast-\theta^\ast_j}=I,    
	\end{equation*}
where the multiplication is the Hadamard product.
Obviously, $I$ has $E_l$ as a component.
Since $l \notin N_i$ for any $i\in \{0,1,\ldots ,d-1\}$, 
we have $l \in N_d$. By Lemma \ref{key}, the desired result follows. 
\end{proof}

\section{Integrality of $K_i$}
	In this section, we consider when 
$K_i=-p_i(d)$ is an integer for each $i \in \{1,2,\ldots,d \}$
for a $Q$-polynomial scheme. 
	The following theorem is important in this section. 
	\begin{theorem}[Suzuki \cite{Suzuki}] \label{Suzuki}
	Let $\mathfrak{X}$ with $m_1>2$ be a $Q$-polynomial scheme with respect to the ordering $\{E_i\}_{i=0}^d$. 
Suppose $\mathfrak{X}$ is $Q$-polynomial with respect to another ordering. Then the new ordering
is one of the following:
\begin{enumerate}
	\item $E_0,E_2,E_4,E_6, \ldots , E_5,E_3,E_1$, 
	\item  $E_0,E_d,E_1,E_{d-1},E_2,E_{d-2},E_3,E_{d-3}, \ldots $, 
	\item  $E_0,E_d,E_2,E_{d-2},E_4,E_{d-4}, \ldots , E_{d-5},E_5,E_{d-3},E_3,E_{d-1},E_1$, 
	\item  $E_0,E_{d-1},E_2,E_{d-3},E_4,E_{d-5}, \ldots , E_5,E_{d-4},E_3,E_{d-2},E_1,E_d$, or 
	\item  $d = 5$ and $E_0,E_5,E_3,E_2,E_4,E_1$.
	\end{enumerate}
	\end{theorem}
	Note that $Q$-polynomial schemes with $m_1=2$
are the ordinary $n$-gons as distance-regular graphs.

	\begin{proposition}
	Let $\mathfrak{X}$ with $m_1>2$ be a $Q$-polynomial association scheme with respect to the ordering $\{ E_i \}_{i=0}^d$. 
If there exists $t$ such that $t\leq d/2$, $t \equiv 1 \pmod 2$ and 
$m_t\ne m_{d-t+1}$, then 
$K_j$ is an integer for any $j$. 
	\end{proposition}
	\begin{proof}
	Let $\mathbb{F}$ be the splitting field of the scheme, 
generated by the entries of the first eigenmatrix $P$. 
Then $\mathbb{F}$ is a Galois extension of the rational field. 
Let $G$ be the Galois group $\Gal (\mathbb{F}/\mathbb{Q})$. 
We consider the action of $G$ on the primitive idempotents 
$E_i$, where elements of $G$ are applied entry-wise.  
Then the action of $G$ on $\{E_i\}_{i=0}^d$ is faithful and $|G| \leq 2$ \cite{Martin-Williford}. 
 
Suppose $K_j$ is not an integer for some $j$. 
Since $-K_j=p_j(d)$ is an eigenvalue of $A_j$, 
$K_j$ is an algebraic integer. 
By the basic number theory, $K_j$ is irrational.  
Therefore $|G|\ne 1$ and hence $|G|=2$. 
Let $\sigma$ be the non-identity element of $G$.  
From the definition of $K_j$, $E_1$ must have an irrational entry, and $E_1^{\sigma} \ne E_1$.  
Therefore $\{E_i^{\sigma} \}_{i=0}^d$ is another $Q$-polynomial ordering with 
the same polynomials $f_i$. 
Hence $\{E_i^{\sigma} \}_{i=0}^d$ coincides with one of (1)--(5) in Theorem \ref{Suzuki}. 

For $d=2$, it is known that $K_i$ is an integer for each $i=1,2$ if 
$m_1 \ne m_2$ \cite{Bannai-Bannai}. 
For (1) and (2) with $d>2$, $(E_1^{\sigma})^{\sigma} \ne E_1$, 
this contradicts that $\sigma^2$ is the identity. 
Since $p_j(d)$ is irrational and $A_j E_d=p_j(d) E_d$, $E_d$ has an irrational entry. 
Therefore  $E_d^{\sigma} \ne E_d$.
For (4), $\sigma$ fixes $E_d$, a contradiction. 
Therefore the ordering $\{E_i^{\sigma}\}_{i=0}^{d}$ coincides with (3) or (5).

Suppose that there exists $t$ such that $t\leq  d/2$, $t \equiv 1 \pmod 2$ and 
$m_t\ne m_{d-t+1}$. 
Since $E_t\circ I= (m_t/|X|) I$, 
we have $E_t^{\sigma} \circ I^{\sigma} = (m_t/|X|) I^{\sigma}$ 
and hence $E_t^{\sigma} \circ I = (m_t/|X|) I \ne (m_{d-t+1}/|X|) I$.  
Therefore $E_t^{\sigma} \ne E_{d-t+1}$. Thus, the ordering 
$\{E_i^{\sigma} \}_{i=0}^d$ does not coincide with (3) for $d \geq 2$. 
If $d=5$, then $m_1 \ne m_5$ and hence $E_1^{\sigma} \ne E_{5}$. Therefore $\{E_i^{\sigma} \}_{i=0}^5$ does not coincide with (5). 
Thus the proposition follows.  
	\end{proof}
	
Remark that the known $Q$-polynomial schemes with some irrational $K_i$ and $d>2$ 
are the ordinary $n$-gons and the association scheme obtained from the icosahedron 
\cite{Kiyota-Suzuki, Martin-Muzychuk-Williford}. 
We can give a similar equivalent condition of the $P$-polynomial property of symmetric association schemes \cite{Kurihara-Nozaki}. 
Let $\theta_i=p_1(i)$ for $0 \leq i \leq d$. 
     \begin{theorem} 
	 Let $\mathfrak{X}$ be a symmetric association scheme of class $d\ge 2$. 
Suppose $\{\theta_j\}_{j=0}^d$ are mutually distinct. Then the following are equivalent: 
\begin{enumerate}
	\item $\mathfrak{X}$ is a $P$-polynomial association scheme with respect to 
$A_1$. 
	\item There exists $l \in \{2,3,\ldots,d\}$ such that  for any $i \in \{1,2,\ldots d\}$, 
	\[
	 \prod_{\substack{j=1\\ j \ne i}}^d \frac{\theta_0-\theta_j}{\theta_i-\theta_j}= -q_i(l). 
	\]
	\end{enumerate}
	Moreover if $(2)$ holds, then $l=i_d$.
	\end{theorem}
	
	\noindent
\textbf{Acknowledgments.} 
Both of authors are supported by the fellowship of the Japan Society for
the Promotion of Science. 
The authors would like to thank Eiichi Bannai, Edwin van Dam, Tatsuro Ito, William J. Martin, Akihiro Munemasa, Hiroshi Suzuki, Makoto Tagami, Hajime Tanaka, Paul Terwilliger and Paul-Hermann Zieschang for useful discussions and comments.  

{\it Hirotake Kurihara}\\
	Mathematical Institute, \\
	Tohoku University \\
	Aramaki-Aza-Aoba 6-3, \\
	Aoba-ku, \\
	Sendai 980-8578, \\
	Japan\\ 
	sa9d05@math.tohoku.ac.jp\\
\quad\\
{\it Hiroshi Nozaki}\\
	Graduate School of Information Sciences, \\
	Tohoku University \\
	Aramaki-Aza-Aoba 6-3-09, \\
	Aoba-ku, \\
	Sendai 980-8579, \\
	Japan\\ 
	nozaki@ims.is.tohoku.ac.jp\\

\end{document}